\pdfoutput=1

\documentclass[11pt, fleqn, a4paper]{article}
\usepackage[utf8]{inputenc}
\usepackage[T1]{fontenc}
\usepackage{geometry}
\usepackage{amsmath,amsfonts,amssymb}
\usepackage{latexsym}
\usepackage{graphicx}
\usepackage{float} 
\usepackage{xcolor}
\usepackage{microtype}
\usepackage{times}
\usepackage{leading}
\usepackage{textcomp}
\usepackage{amsthm}
\usepackage{thmtools}
\usepackage{thm-restate}
\usepackage{pifont}
\usepackage[all]{xy}
\usepackage{enumitem}
\usepackage[hang, flushmargin]{footmisc}
\usepackage[initials, lite, msc-links, nobysame]{amsrefs}

\leading{15pt}

\setitemize{noitemsep, topsep=0pt, listparindent=\parindent, parsep=0pt}
\setlist{listparindent=\parindent, parsep=0pt, topsep=0pt, itemsep=0pt}

\newtheorem{thm}{Theorem}[section]
\newtheorem{cor}[thm]{Corollary}
\newtheorem{lem}[thm]{Lemma}
\newtheorem{prop}[thm]{Proposition}
\newtheorem{conj}[thm]{Conjecture}
\newtheorem{defn}[thm]{Definition}
\newtheorem{ex}[thm]{Example}

\makeatletter                         
\providecommand*{\toclevel@thm}{0}
\providecommand*{\toclevel@lem}{0}
\providecommand*{\toclevel@cor}{0}
\providecommand*{\toclevel@prop}{0}
\providecommand*{\toclevel@defn}{0}
\providecommand*{\toclevel@conj}{0}
\providecommand*{\toclevel@ex}{0}
\makeatother  

\DeclareMathOperator\ad{ad}
\DeclareMathOperator\Ann{Ann}
\DeclareMathOperator\Aut{Aut}
\DeclareMathOperator\Ext{Ext}

\DeclareMathOperator\gld{gld}
\DeclareMathOperator\gr{gr}
\DeclareMathOperator\pd{pd}
\DeclareMathOperator\pid{PI-deg}

\DeclareMathOperator\tr{trace}

\newcommand{\B}{\mathcal{B}}

\newcommand{\cnt}{\mathcal{Z}}

\newcommand{\F}{\mathcal{F}}
\newcommand{\gk}{\text{GK.dim}}

\newcommand{\id}{\mathrm{id}}

\newcommand{\inv}{^{-1}}
\newcommand{\iso}{\cong}

\newcommand{\niso}{\ncong}
\newcommand{\K}{K}
\renewcommand{\L}{\mathcal{L}}

\newcommand{\tensor}{\otimes}
\newcommand{\van}{\mathcal{V}}

\newcommand{\xx}{R_{x^2}}
\newcommand{\yx}{R_{yx}}
\newcommand{\qp}{\mathcal{O}_q(\K^2)}

\newcommand{\fwa}{A_1(\K)}
\newcommand{\wa}{A_1^q(\K)}

\newcommand{\jp}{\mathcal{J}}
\newcommand{\sjp}{\mathcal{J}_1}
\newcommand{\os}{\mathcal{S}}
\newcommand{\env}{\mathfrak{U}}

\newcommand{\kx}{\K[x]}
\newcommand{\sxx}{R_{x^2-1}}

\newcommand{\hsjp}{H(\sjp)}

\renewcommand{\O}{\mathcal{O}}

\title{PBW deformations of Artin-Schelter regular algebras}
\author{Jason Gaddis}
\date{}

\begin{document}
\maketitle

\begin{abstract}We consider algebras that can be realized as PBW deformations of (Artin-Schelter) regular algebras. This is equivalent to the homogenization of the algebra being regular. It is shown that the homogenization, when it is a geometric algebra, contains a component whose points are in 1-1 correspondence with the simple modules of the deformation. We classify all PBW deformations of 2-dimensional regular algebras and give examples of 3-dimensional deformations. Other properties, such as the skew Calabi-Yau property and closure under tensor products, are considered.\end{abstract}

\section{Introduction}

In \cite{artsch}, Artin and Schelter introduced a class of 3-dimensional graded algebras which may be regarded as noncommutative versions of the polynomial ring in three variables. These algebras, which posses suitably nice growth and homological properties, are now called (Artin-Schelter) regular algebras. In \cite{atv}, Artin, Tate, and Van den Bergh showed that these algebras have geometric interpretations related to point schemes in projective space. The classification of 3-dimensional regular algebras was completed in the two aforementioned articles and the classification of higher dimensional regular algebras is an active open problem.

We consider algebras on the periphery of this classification. Indeed, there are many non-graded algebras with similar properties to regular algebras. Examples include the quantum Weyl algebras and the enveloping algebra of the two-dimensional solvable Lie algebra. If $A$ is an algebra such that its homogenization, $H(A)$, is regular, we say $A$ is \textit{essentially regular}. This is equivalent to the associated graded ring, $\gr(A)$, being regular and equivalent to $A$ being a \textit{PBW deformation} of $\gr(A)$ (Proposition \ref{grprop}). If $x_0$ is the homogenizing element, then one can pass certain properties between $A$ and $H(A)$ (Propositions \ref{cntr} and \ref{hprops}) via the isomorphism $H(A)[x_0\inv] \iso A[x_0^{\pm 1}]$. In Section \ref{cy}, we review the definition of a skew Calabi-Yau algebra and show that all essentially regular algebras are skew Calabi-Yau.

All 3-dimensional regular algebras are geometric algebras, that is, they are canonically identified with a pair $(E,\sigma)$ where $E \subset \mathbb{P}(V^*)$ is a scheme and $\sigma \in \Aut(E)$. If $A$ is essentially regular and $H(A)$ is geometric, then there exists a subscheme $E_1 \subset E$ whose points are in 1-1 correspondence with the 1-dimensional simple modules of $H(A)$ (Proposition \ref{dim1}). In case $A$ is 2-dimensional and not PI, these are all of the finite-dimensional simple modules of $H(A)$ (Proposition \ref{fdim}). We generalize this result in Proposition \ref{Itors} and make progress towards a conjecture by Walton that all finite-dimensional simple modules of a non-PI deformed Sklyanin algebra are 1-dimensional \cite{walton}. 

Relying on work done in \cite{twogen}, we classify all 2-dimensional essentially regular algebras (Corollary \ref{eclass}) and compile some known examples of 3-dimensional essentially regular algebras (Examples \ref{ex1}-\ref{ex3}). Finally, in Section \ref{fivedim}, we show that the property of being essentially regular is closed under tensor products and use this to give examples of 5-dimensional regular algebras.

\section{Definitions and initial properties}\label{props}

Let $\K$ be an algebraically closed, characteristic zero field. All algebras are $\K$-algebras and all isomorphisms should be read as `isomorphisms of $\K$-algebras'. Suppose $A$ is defined as a factor algebra of the free algebra on $n$ generators by $m$ polynomial relations, i.e., 
\begin{align}\label{form}
	A = \K\langle x_1,\hdots,x_n \mid f_1,\hdots,f_m \rangle.
\end{align}
Throughout, we assume the generators $x_i$ all have degree 1.

Given a (noncommutative) polynomial $f \in \K\langle x_1,\hdots,x_n\rangle$ with $\alpha=\deg(f)$, write 
\begin{align}\label{fexpl}
	f = \sum_{\gamma \in \Gamma} c_\gamma x_{i_1}^{\alpha_{\gamma_1}} \cdots x_{i_\ell}^{\alpha_{\gamma_\ell}}, c_\gamma \in \K, \alpha_{\gamma_i} \in \K, \sum_{i=1}^\ell \alpha_{\gamma_i} \leq \alpha,
\end{align} 
where $\Gamma$ ranges over distinct monomials in the free algebra and all but finitely many of the $c_\gamma$ are zero. The homogenization of $f$ is then
\begin{align*}
	\hat{f} = \sum_{\gamma \in \Gamma} c_\gamma x_{i_1}^{\alpha_{\gamma_1}} \cdots x_{i_\ell}^{\alpha_{\gamma_\ell}} x_0^{\alpha_{\gamma_0}},
\end{align*}
where $x_0$ is a new, central indeterminate and $\alpha_{\gamma_0}$ is chosen such that $\sum_{i=0}^\ell \alpha_{\gamma_i} = \alpha$. Then $\hat{f}$ is homogeneous.

\begin{defn}\label{hom-def}Let $A$ be of form (\ref{form}). The \textbf{homogenization} $H(A)$ of $A$ is the $\K$-algebra on the $n+1$ generators $x_0,x_1,\hdots,x_n$ subject to the homogenized relations $\hat{f}_i$ as well as the additional relations $x_0x_i-x_ix_0$ for all $i \in \{1,\hdots,n\}$.\end{defn}

A \textit{filtration} $\F$ on an algebra $A$ is a collection of vector spaces $\{\F_n(A)\}$ such that \[\F_n(A) \subset \F_{n+1}(A), \F_n(A) \cdot \F_m(A) \subset \F_{n+m}(A) \text{ and } \bigcup \F_n(A) = A.\] The filtration $\F$ is said to be \textit{connected} if $\F_0(A)=\K$ and $\F_\ell(A) = 0$ for all $\ell < 0$. The \textit{associated graded algebra} of $A$ is $\gr_\F(A) := \bigoplus_{i \geq 0} \F_i(A)/\F_{i-1}(A)$. The algebra $\gr_\F(A)$ is said to be connected if the filtration $\F$ is connected. 

Associated to the pair $(A,\F)$ is also the \textit{Rees ring} of $A$,
	\[ R_\F(A) := \bigoplus_{n \geq 0} \F_n(A) x_0^n.\] 
For an algebra defined by generators and relations, as above, there is a standard connected filtration wherein $\F_\ell(A)$ is the span of all monomials of degree at most $\ell$. Since this filtration is the only one we consider, we drop the subscript on $\gr(A)$ and $R(A)$. One can always recover $A$ and $\gr(A)$ from $H(A)$ via $A \iso H(A)/(x_0-1)H$ and $\gr(A) \iso H(A)/x_0H$, respectively.

An algebra is said to be \textit{graded} if $\gr(A) = A$. In this case, we write $A_m$ for the vector space spanned by homogeneous elements of degree $m$. If $A$ is graded, then the \textit{global dimension} of $A$, $\gld(A)$, is the projective dimension of the trivial module $\K_A = A/A_+$, where $A_+$ is the augmentation ideal generated by all degree one elements. Let $V$ be a $\K$-algebra generating set for $A$ and $V^n$ the set of degree $n$ monomials in $A$. The Gelfand-Kirilov (GK) dimension of $A$ is defined as
	 \[ \gk(A) := \limsup_{n \rightarrow \infty} \log_n (\dim V^n). \]
The algebra $A$ is said to be \textit{AS-Gorenstein} if $\Ext_A^i(\K_A,A)\iso \delta_{i,d} {_A}\K$ where $\delta_{i,d}$ is the Kronecker delta and $d=\gld(A)$.

\begin{defn}A connected graded algebra is said to be (Artin-Schelter) \textbf{regular} of dimension $d$ if $H$ has finite global dimension $d$, finite GK dimension, and is AS-Gorenstein. We say an algebra $A$ is \textbf{essentially regular} of dimension $d$ if $H(A)$ is regular of dimension $d+1$.\end{defn}

In the case that $x_0$ is not a zero divisor, we have $H(A) \iso R(A)$ (\cite{wu}, Proposition 2.6) and $H(A)$ becomes a regular central extension of $\gr(A)$ (see \cite{casshel}, \cite{bruyn}). However, this need not always be the case, which the next example illustrates, and so we choose to use $H(A)$ instead of $R(A)$ in the above definition.

\begin{ex}Let $A=\K\langle x_1,x_2 \mid x_1^2-x_2\rangle$. Then $A \iso \K[x]$. However, the algebra $H(A)$ is not regular. Indeed,
	\[ x_1x_2x_0 = x_1x_1^2 = x_1^2x_1 = x_2x_0x_1 = x_2x_1x_0 \Rightarrow (x_1x_2-x_2x_1)x_0=0.\]
Thus, either $H(A)$ is not a domain or else $H(A)$ is commutative. The latter cannot hold because $H(A)/x_0H(A) \iso \K\langle x_1,x_2 \mid x_1^2\rangle$ is not commutative. By \cite{atv2}, Theorem 3.9, all regular algebras of dimension at most four are domains. Hence, in this case, $H(A) \niso R(A)$.\end{ex}

If $A$ is essentially regular of dimension at most three, then $H(A)$ is domain. We will assume, hereafter, that $x_0$ is not a zero divisor in $H(A)$ when $A$ is essentially regular.

The dimension of an essentially regular algebra is not the same as its global dimension in all cases. The first Weyl algebra, $\fwa = \K\langle x,y \mid yx-xy+1 \rangle$, is dimension two essentially regular but has global dimension one. It would be interesting to know whether there is a lower bound on the global dimension of an essentially regular algebra based on that of its homogenization.

The following lemma is useful in passing properties between $A$ and $H(A)$.

\begin{lem}\label{loclem}Suppose $x_0$ is not a zero divisor. Then $H(A)[x_0\inv] \iso A[x_0^{\pm 1}]$.\end{lem}

\begin{proof}Let $f$ be a defining relation for $A$ and $\hat{f}$ the homogenized relation in $H$. Let $\alpha=\deg(f)$ and write $f$ as in (\ref{fexpl}). Then \[0 = x_0^{-\alpha}\hat{f} = \sum_{\gamma \in \Gamma} c_\gamma (x_0\inv x_{i_1})^{\alpha_{\gamma_1}} \cdots (x_0\inv x_{i_\ell})^{\alpha_{\gamma_\ell}}.\] If we let $X_0=x_0$ and $X_i=x_0\inv x_i$ for $i > 0$, then the $\{X_i\}_{i \geq 0}$ generate $A[x_0^{\pm 1}]$ in $H[x_0\inv]$. Conversely, in $A[x_0^{\pm 1}]$ we have \[0 = x_0^\alpha f = \sum_{\gamma \in \Gamma} c_\gamma (x_0x_{i_1})^{\alpha_{\gamma_1}} \cdots (x_0x_{i_\ell})^{\alpha_{\gamma_\ell}}.\] If we let $X_0=x_0$ and $X_i=x_0x_i$ for $i > 0$, then the $\{X_i\}_{i \geq 0}$ generate $H[x_0\inv]$ in $A[x_0^{\pm 1}]$.\end{proof}

Let $\cnt(A)$ denote the center of $A$. One would expect a natural equivalence between the center of a homogenization and the homogenization of a center. The next proposition formalizes that idea. 

\begin{prop}\label{cntr}Suppose $x_0$ is not a zero divisor. By identify generators, we have \[\cnt(H(A)) = H(\cnt(A)).\]\end{prop}

\begin{proof}By \cite{rowen}, Propositions 1.2.20 (ii) and 1.10.13, along with Lemma \ref{loclem},
	\[ H(\cnt(A))[x_0\inv] \iso \cnt(A)[x_0^{\pm 1}] = \cnt(A[x_0^{\pm 1}]) \iso \cnt(H(A)[x_0\inv]) = \cnt(H(A))[x_0\inv].\]
Thus, $H(\cnt(A))[x_0\inv] \iso \cnt(H(A))[x_0\inv]$. It remains to be shown that the subalgebras $H(\cnt(A))$ and $\cnt(H(A))$ are isomorphic and, moreover, the elements can be identified by generators.

Let $\hat{f} \in H(\cnt(A))$, then $\hat{f} \in H(\cnt(A))[x_0\inv]$ and, by Lemma \ref{loclem}, $f \in \cnt(A)[x_0^{\pm 1}]$. Since only positive powers of $x_0$ appear in $f$, then $f \in \cnt(A)[x_0]$. Therefore, \[f \in  \cnt(A)[x_0^{\pm 1}] = \cnt(A[x_0^{\pm 1}]),\] and so $\hat{f} \in \cnt(H(A)[x_0\inv])=\cnt(H(A))[x_0\inv]$. Again, since only positive powers of $x_0$ appear in $\hat{f}$, then $\hat{f} \in \cnt(H(A))$. The converse is similar.
\end{proof}

We now consider properties that pass between an essentially regular algebra and its homogenization.

\begin{prop}\label{hprops}Let $A$ be essentially regular and $H=H(A)$.
\begin{enumerate}
  \item $A$ is prime if and only if $H$ is prime;
  \item $A$ is PI if and only if $H$ is PI;
  \item $A$ is noetherian if $H$ is noetherian; 
  \item $H$ is noetherian if $\gr(A)$ is noetherian;
  \item $H$ is not primitive.
\end{enumerate}\end{prop}

\begin{proof}(1) is well-known since $x_0$ is central and not a zero divisor. (2) is a consequence of Proposition \ref{cntr}. (3) is clear because $A$ is a factor algebra of $H$ and (4) follows from \cite{atv}, Lemma 8.2. The algebra $H$ is affine over the uncountable, algebraically closed field $\K$, so (5) follows from \cite{kirkkuz}, Proposition 3.2.\end{proof}

Let $F$ be the set of relations defining $A$ and let $R$ be the relations $F$ filtered by degree. There is a canonical surjection $\gr(A) \rightarrow B :=  \K\langle x_1,\hdots x_n \mid R\rangle$. We say $A$ is a \textit{Poincar\'{e}-Birkhoff-Witt (PBW) deformation} of $\gr(A)$ if that map is an isomorphism. The next proposition shows that essentially regular algebras are equivalent to PBW deformations of regular algebras.

\begin{prop}\label{grprop}An algebra $A$ is essentially regular if and only if $\gr(A)$ is regular. Moreover, if $A$ is essentially regular, then it is a PBW deformation of $\gr(A)$.\end{prop}

\begin{proof}Let $B=\gr(A)$ and $H=H(A)$. Since $x_0\in H$ is central and not a zero divisor, then by the Rees Lemma (\cite{rotman}, Theorem 8.34), $\Ext_B^n(\K_B,B) \iso \Ext_{H}^{n+1}(\K_H,H)$. Hence, $B$ is AS-Gorenstein if and only if $H$ is. Moreover, since $B$ (resp. $H$) is graded, then $\gld(B)=\pd(\K_B)$ (resp. $\gld(H)=\pd(\K_H)$). By the Rees Lemma, $\gld(B)=d$ if and only if $\gld(H)=d+1$. The sequence $0 \rightarrow x_0H \rightarrow H \rightarrow B \rightarrow 0$ is exact, so $\gk(B) \leq \gk(H)-1 < \infty$ when $H$ is regular. Conversely, if $B$ is regular, then $\gk(A)=\gk(B) < \infty$. Localization at the central regular element $x_0$ in $H$ and in $A[x_0]$ preserves GK dimension (\cite{mcrob}, Proposition 8.2.13). This, combined with Lemma \ref{loclem}, gives, \[\gk(H) = \gk(H[x_0\inv]) = \gk(A[x_0^{\pm 1}]) = \gk(A)+1 < \infty.\]
That $A$ is a PBW deformation now follows from \cite{cassidy}, Theorem 1.3.\end{proof}

\begin{cor}\label{essAS}If $A$ is regular, then $A$ is essentially regular.\end{cor}

\begin{cor}If $A$ is a noetherian essentially regular algebra, then $A$ has finite global and GK dimension.\end{cor}

\begin{proof}By \cite{mcrob}, Corollary 6.18, and because $\gr(A)$ is regular, $\gld A \leq \gld \gr(A) < \infty$. The statement on GK dimension follows from the proof of Proposition \ref{grprop}.\end{proof}

\begin{cor}\label{polyext}Let $\xi$ be a central indeterminate over an algebra $A$. The polynomial ring $A[\xi]$ is essentially regular of dimension $d$ if and only if $A$ is essentially regular of dimension $d-1$.\end{cor}

\begin{proof}We need only observe that $H(A[\xi])=H(A)[\xi]$ and that regularity is preserved under polynomial extensions.\end{proof}

In the next section we observe that essential regularity is preserved under certain skew polynomial extensions.

We end this section with a classification of examples of essentially regular algebras in dimension two and several examples of those of dimension three. If $A$ is essentially regular of dimension two, then $H=H(A)$ is dimension three regular. Hence, $H$ either has three generators subject to three quadratic relations, or else it has two generators subject to two cubic relations. Since $H$ is a homogenization, then the commutation relations of $x_0$ give two quadratic relations, so there must be some presentation in the first form. Since $A \iso H/(x_0-1)H$, then the commutation relations drop off and we are left with one quadratic relation. Thus, to determine the 2-dimensional essentially regular algebras, we begin with a list of all algebras defined by two generators subject to a single quadratic relation. An easy consequence is that a subset of this list contains the 2-dimensional essentially regular algebras.

\begin{thm}[\cite{twogen}]\label{classification} Suppose $A \iso \K\langle x,y \mid f \rangle$ where $f$ is a polynomial of degree two. Then $A$ is isomorphic to one of the following algebras:
\begin{align*}
  &\qp, f=xy-qyx ~ (q \in \K^\times), & ~ & \wa, f=xy-qyx+1 ~ (q \in \K^\times), \\
  &\jp,  f=yx-xy+y^2, & ~ & \sjp, f=yx-xy+y^2+1, \\
  &\env, f=yx-xy+y, & ~ & \kx, f= x^2-y, \\
  &\xx, f=x^2, & ~ & \sxx, f=x^2-1,\\
  &\yx, f=yx, & ~ & \os, f=yx-1.
\end{align*}
Furthermore, the above algebras are pairwise non-isomorphic, except \[\qp \iso \mathcal{O}_{q\inv}(\K^2) \text{ and } \wa \iso A_1^{q\inv}(\K).\]\end{thm}

\begin{cor}\label{eclass}The dimension two essentially regular algebras are \[ \qp, \wa, \jp, \sjp, \env.\]\end{cor}

\begin{proof}The algebras $\qp$ and $\jp$ are 2-dimensional regular \cite{artsch}. On the other hand, $\xx$ and $\yx$ are not domains and therefore not regular \cite{atv2}. Therefore, $\wa$, $\sjp$ and $\env$ are essentially regular of dimension two whereas $\kx$, $\sxx$ and $\os$ are not by Proposition \ref{grprop}.\end{proof}

In the following, we collect a few examples of 3-dimensional essentially regular algebras. We have already observed that if $H$ is 3-dimensional regular, then $H$ is 3-dimensional essentially regular (Corollary \ref{essAS}) and if $A$ is 2-dimensional regular, then $A[\xi]$ is 3-dimensional essentially regular (Corollary \ref{polyext}).

\begin{ex}(\cite{lbvdb})\label{ex1} Let $L$ be the Lie algebra over $\K$ generated by $\{x_1,x_2,x_3\}$ subject to the relations $[x_i,x_j] = \sum_{k=1}^n \alpha_{ij,k} x_k$. The enveloping algebra $U(L)$ is 3-dimensional essentially regular. A special case of this is when $L=\mathfrak{sl}_2(\mathbb{C})$ (see \cite{lesmith}).\end{ex}

\begin{ex}(\cite{cassidy})\label{ex2} Essentially regular algebras need not be skew polynomial rings. The down-up algebra $A(\alpha,\beta,\gamma)$ for $\alpha,\beta,\gamma \in \K$ is defined as the $\K$-algebra on generators $d,u$ subject to the relations $d^2u = \alpha dud + \beta ud^2 + \gamma d$, $du^2 = \alpha udu + \beta u^2 d + \gamma u$. The algebra $A(\alpha,\beta,\gamma)$ is 3-dimensional essentially regular if and only if $\beta \neq 0$. More generally, PBW deformations of generic cubic regular algebras were computed by Fl{\o}ystad and Vatne in \cite{flovat}.\end{ex}

\begin{ex}\label{ex3}Let $A$ be an algebra with generators $x_0,x_1,x_2$ and let $r_i=x_i x_j - q x_j x_i + s(x_k)$, $\deg s \leq 2$, $q \in \K^\times$, where $i=0,1,2$ and $j\equiv i+1 \mod 3, k \equiv i+2 \mod 3$. Since $\gr(A)$ is a Sklyanin algebra, then $A$ is 3-dimensional essentially regular by Proposition \ref{grprop}.\end{ex}

\section{Skew Calabi-Yau algebras}\label{cy}

Closely related to regular algebras is the more recent notion of a Calabi-Yau and skew Calabi-Yau algebra. We define (skew) Calabi-Yau algebras and show that essentially regular algebras have this property. This is then used to show that the property of essential regularity is preserved under certain skew polynomial extensions.

The \textit{enveloping algebra} of $A$ is defined as $A^e:=A \tensor A^{op}$. If $M$ is both a left and right $A$-module, then $M$ is an $A^e$ module with the action given by $(a \tensor b)\cdot x = axb$ for all $x \in M$, $a,b \in A$. Correspondingly, given automorphisms $\sigma,\tau \in \Aut(A)$, we can define the twisted $A^e$-module ${^\sigma}M^\tau$ via the rule $(a \tensor b)\cdot x = \sigma(a)x\tau(b)$ for all $x \in M$, $a,b \in A$. When $\sigma$ is the identity, we omit it.

\begin{defn}An algebra $A$ is said to be \textbf{homologically smooth} if it has a finite resolution by finitely generated projectives as an $A^e$-module. The length of this resolution is known as the Hochschild dimension of $A$.\end{defn}

The Hochschild dimension of $A$ is known to coincide with the global dimension of $A$ (\cite{berger}, Remark 2.8). 

\begin{defn}An algebra $A$ is said to be \textbf{skew Calabi-Yau} of dimension $d$ if it is homologically smooth and there exists an automorphism $\tau \in \Aut(A)$ such that there are isomorphisms
\[ \Ext_{A^e}^i(A,A^e) \iso \begin{cases}0 & \text{ if } i\neq d \\ A^\tau & \text{ if } i = d.\end{cases}\]
If $\tau$ is the identity, then $A$ is said to be \textbf{Calabi-Yau}.\end{defn}

The condition on $\Ext$ in the above definition is sometimes referred to as the \textit{rigid Gorenstein} condition \cite{brzha}.

Regular algebras are important examples of (skew) Calabi-Yau algebras. This idea was formalized recently in \cite{rrz}.

\begin{thm}[Reyes, Rogalski, Zhang]\label{rrz}An algebra is connected graded Calabi-Yau if and only if it is regular.\end{thm}

On the other hand, (skew) Calabi-Yau algebras need not be graded. For example, the enveloping algebra of any finite-dimensional Lie algebra $\mathfrak{g}$ is Calabi-Yau if $\tr(\ad_\mathfrak{g}(x))$ for all $x \in \mathfrak{g}$ (\cite{he}, Lemma 4.1). PBW deformations of Calabi-Yau algebras were studied by Berger and Taillefer \cite{berger}. Their main result is that PBW deformations of Calabi-Yau algebras defined by quivers and potentials are again Calabi-Yau. More recently, Wu and Zhu proved that a PBW deformation of a noetherian Koszul Calabi-Yau algebra is Calabi-Yau if and only if its Rees ring is \cite{wu}. We generalize their result partially below.

\begin{thm}\label{cyprop}If $A$ is essentially regular, then $A$ is skew Calabi-Yau.\end{thm}

\begin{proof}By Proposition \ref{grprop}, $A$ is a PBW deformation of the regular algebra $\gr(A)$. By \cite{yekzha}, $A$ has a rigid dualizing complex $R=A^\sigma[n]$ for some integer $n$ and some $\sigma \in \Aut(A)$. This is precisely the condition for $A$ to be rigid Gorenstein. Since $\gr(A)$ is regular (Proposition \ref{grprop}), then $\gr(A)$ is Calabi-Yau by Theorem \ref{rrz}. Thus, $\gr(A)$ is homologically smooth and so, by \cite{mcrob}, Theorem 7.6.17, $A$ is homologically smooth.\end{proof}

\begin{cor}If $A$ is a PBW deformation of a graded skew Calabi-Yau algebra. Then $A$ is skew Calabi-Yau.\end{cor}

\begin{proof}By hypothesis, $\gr(A)$ is graded skew Calabi-Yau. Hence, $\gr(A)$ is regular by Theorem \ref{rrz} and so $A$ is essentially regular. Hence, $A$ is skew Calabi-Yau by Theorem \ref{cyprop}.\end{proof}

Let $R$ be a ring, $\sigma \in Aut(R)$ and $\delta$ a $\sigma$-derivation, that is, $\delta:R \rightarrow R$ satisfies the twisted Leibniz rule, $\delta(rs) = \sigma(r)\delta(s) + \delta(r)s$ for all $r,s \in R$. The \textit{skew polynomial ring} (or Ore extension) $R[\xi;\sigma,\delta]$ is defined via the commutation rule $\xi r = \alpha(r)\xi + \delta(r)$ for all $r \in R$. 

Suppose $A$ is of the form \eqref{form}. If $\sigma \in \Aut(A)$ with $\deg(\sigma(x_i)) = 1$, then $\sigma$ lifts to an automorphism $\hat{\sigma} \in \Aut(H(A))$ defined by $\hat{\sigma}(x_0)=x_0$ and $\hat{\sigma}(x_i)=\widehat{\sigma(x_i)}$ for $i > 0$. To see this, let $g$ be a defining relation for $A$ and $\hat{g}$ the corresponding relation in $H(A)$. For a generator $x_i$ of $A$, $\sigma(x_i)=y_{i,1} + y_{0,1}$ for some $y_{i,1} \in A_1$ and $y_{0,1} \in A_0=\K$. Thus, $\hat{\sigma}(x_i)=y_{i,1}+y_{0,1}x_0$. We must show that this rule implies $\hat{\sigma}(g) = \widehat{\sigma(g)}$. Suppose $\deg(g)=d$ and write $g = \sum_{i=0}^d g_i$ with $\deg(g_i)=i$. Then $\sigma(g_i) = \sum_{j=0}^i g_{i,j}$ where $\deg(g_{i,j})=j$. Thus,
	\[ \widehat{\sigma(g)} = \widehat{\sum_{i=0}^d \sigma(g_i)} = \widehat{\sum_{i=0}^d \sum_{j=0}^i g_{i,j}} = \sum_{i=0}^d \sum_{j=0}^i g_{i,j} x_0^{d-j}.\]
Now $\hat{g} = \sum_{i=0}^d g_i x_0^{d-i}$ and a similar computation shows
	\[ \hat{\sigma}(\hat{g}) = \sum_{i=0}^d \hat{\sigma}(g_i) x_0^{d-i} = \sum_{i=0}^d \sum_{j=0}^i (g_{i,j} x_0^{i-j}) x_0^{d-i} = \sum_{i=0}^d \sum_{j=0}^i g_{i,j} x_0^{d-j}.\]

Similarly, if $\delta$ is a $\sigma$-derivation of $A$ with $\delta(x_i)\leq 2$, then $\hat{\delta}$ is $\hat{\sigma}$-derivation of $H(A)$ with $\hat{\delta}(x_0)=0$ and $\hat{\delta}(x_i)=\widehat{\delta(x_i)}$ for $i>0$.

\begin{lem}\label{hsp}Let $A$, $\sigma$, and $\delta$ be as above. Then $H(A[\xi;\sigma,\delta]) = H(A)[\xi;\hat{\sigma},\hat{\delta}]$.\end{lem}

\begin{proof}Let $f_1,\hdots,f_m$ be the defining relations for $A$. The defining relations for $A[\xi;\sigma,\delta]$ are then $f_1,\hdots,f_m$ along with $e_1,\hdots,e_n$ where $e_i=x_i\xi-\sigma(\xi) x_i - g_i$. The defining relations for $H(A[\xi;\sigma,\delta])$ are then $\hat{f}_i,\hdots,\hat{f}_m$ along with $\hat{e}_i = x_i\xi-\hat{\sigma}(\xi) x_i - \hat{g}_i$. By Lemma \ref{hsp}, these are precisely the defining relations for $H(A)[\xi;\hat{\sigma},\hat{\delta}]$.\end{proof}

\begin{prop}Let $A$ be essentially regular. If $\sigma$ and $\delta$ are as above, then $A[\xi;\sigma,\delta]$ is essentially regular.\end{prop}

\begin{proof}Let $R = H(A)[\xi;\hat{\sigma},\hat{\delta}]$. By Lemma \ref{hsp}, it suffices to prove that $R$ is regular. Since $H(A)$ is regular, then it is Calabi-Yau. By \cite{lww}, Theorem 3.3, skew polynomial extensions of Calabi-Yau algebras are Calabi-Yau and so $R$ is Calabi-Yau. Moreover, $\hat{\sigma}$ and $\hat{\delta}$ preserve the grading on $H(A)$ and so $R$ is graded. Thus, by Theorem \ref{rrz}, $R$ is regular.\end{proof}

If $\delta=0$ (so $\hat{\delta}=0$), then $\xi$ is a normal element in $H(A)[\xi;\hat{\sigma}]$. Thus, by the Rees Lemma and \cite{mcrob}, Proposition 7.2.2, it follows that $A[\xi;\sigma]$ essentially regular implies $A$ is essentially regular. It is not clear if this holds in the case $\delta \neq 0$.

\section{Geometry of Homogenized Algebras}\label{geom}

In \cite{atv}, Artin, Tate, and Van den Bergh showed that every dimension two and dimension three regular algebra surjects onto a \textit{twisted homogeneous coordinate ring}. We begin this section by defining a twisted homogeneous coordinate ring following the exposition in \cite{keeler}. We then go on to define the related concept of a \textit{geometric algebra}, which was originally called an \textit{algebra defined by geometric data} by Vancliff and Van Rompay \cite{vanvan}. 

While one would not expect such a construction for deformations of regular algebras, one might hope to recover information about the deformed algebra from the geometry associated to the homogenization of a deformation. We show that, in certain cases, this geometry allows us to classify all finite-dimensional simple modules of a deformed regular algebra.

Let $E$ be a projective scheme and $\sigma \in \Aut(E)$. Set $\L_0=\O_E$, $\mathcal{L}_1 = \mathcal{L}$, and \[\L_d = \L \tensor_{\O_E} \L^\sigma \tensor_{\O_E} \cdots \tensor_{\O_E} \L^{\sigma^{d-1}} \text{ for } d \geq 2.\] Define the (graded) vector spaces $\B_m = H^0(E,\mathcal{L}_m)$. Taking global sections of the natural isomorphism $\L_d \tensor_{\O_E} \mathcal{L}_e^{\sigma^d} \iso \L_{d+e}$ gives a multiplication defined by $a \cdot b = a\sigma^m(b) \in \B_{m+n}$ for $a \in \B_m$, $b \in \B_n$. Thus, if $\sigma = \id_E$, then this construction defines the the (commutative) homogeneous coordinate ring of $E$.

\begin{defn}The \textbf{twisted homogeneous coordinate ring} of $E$ with respect to $\mathcal{L}$ and $\sigma$ is the $\mathbb{N}$-graded ring $\B=\B(E,\L,\sigma):=\bigoplus_{d \geq 0} H^0(E,\mathcal{L}_d)$ with multiplication defined as above.\end{defn}

Artin and Stafford have shown that every domain of GK dimension two is isomorphic to a twisted homogeneous coordinate ring \cite{artstaf}. Hence, if $H$ is 3-dimensional regular and $g \in H_3$ is a normal element which is not a zero-divisor, then $H/gH$ must be isomorphic to some $\B$. By (\cite{atv}, Theorem 6.8), every 3-dimensional regular algebra contains such an element, though it may be zero.

To define a geometric algebra, we make a slight change of notation to conform to convention. In addition, we specialize to the case of quadratic algebras. These algebras were originally defined by Vancliff and Van Rompay. They were renamed \textit{geometric algebras} by Mori \cite{mori} and we use that definition here.

The free algebra $\K\langle x_0,x_1,\hdots,x_n\rangle$ is equivalent to the tensor algebra $T(V)$ on the vector space $V=\{x_0,\hdots,x_1\}$. If $H$ is quadratic and homogeneous, then we write $H=T(V)/(R)$ where $R$ is the set of defining polynomials of $H$. Any defining polynomial may be regarded as a bilinear form $f:V \tensor_k V \rightarrow \K$. Write $f = \sum \alpha_{ij} x_i \tensor x_j$, $\alpha_{ij} \in \K$. If $p,q \in \mathbb{P}(V^*)$, written as $p=(a_0:a_1:\cdots:a_n)$ and $q=(b_0:b_1:\cdots:b_n)$, then $f(p,q) = \sum \alpha_{ij} a_ib_j$. Define the \textit{vanishing set} of $R$ to be \[ \van(R) = \{ (p,q) \in \mathbb{P}(V^*) \times \mathbb{P}(V^*) \mid f(p,q)=0 \text{ for all } f \in R\}.\]
	
\begin{defn}A homogeneous quadratic algebra $H=T(V)/(R)$ is called \textbf{geometric} if there exists a scheme $E \subset \mathbb{P}(V^*)$ and $\sigma \in \Aut E$ such that
\begin{align*}
  \textbf{G1} ~ &\van(R) = \{(p,\sigma(p)) \in \mathbb{P}(V^*) \times \mathbb{P}(V^*) \mid p \in E\}, \\
  \textbf{G2} ~ &R = \{f \in V \tensor_k V \mid f(p,\sigma(p)) = 0 \text{ for all } p \in E\}.
\end{align*}
The pair $(E,\sigma)$ is called the \textbf{geometric data} corresponding to $H$.\end{defn}

All regular algebras of dimension $d \leq 3$ are geometric. The classification of quadratic regular algebras given in \cite{atv} shows that either $E=\mathbb{P}^2$, or else $E$ is a cubic divisor in $\mathbb{P}^2$. The projective scheme $E$ is referred to as the point scheme of $H$. It is not true that every 4-dimensional regular algebra is geometric \cite{rompay}. However, it seems that they are in the generic case.

\begin{prop}\label{geo}If $A$ is a PBW deformation of a quadratic geometric algebra, then $H(A)$ is geometric.\end{prop}

\begin{proof}Let $V=\{x_1,\hdots,x_n\}$ and $W=\{x_0\} \cup V$. Write \[H(A)=T(W)/(R) \text{ and } \gr(A)=T(V)/(S).\] Choose $p,q \in \van(R)$ and write \[p=(a_0:a_1:\cdots:a_n), ~~q=(b_0:b_1:\cdots:b_n).\] Let $e_i=x_0x_i-x_ix_0$, $i=1,\hdots,n$, be the commutation relations of $x_0$ in $H$. Suppose $a_0=0$, then $e_i(p,q)=0$ implies $a_ib_0=0$ for $i=1,\hdots,n$. Since $q$ is not identically zero, then $b_0=0$. Reversing the argument, we see that $a_0=0$ if and only if $b_0=0$. Let $E_0 \subset \van(R)$ be those points with the first coordinate zero and define $\sigma{\mid_{E_0}}$ to be the automorphism corresponding to $\gr(A)$. Then $E_0$ is $\sigma$-invariant and the restriction $\left(E_0,\sigma{\mid_{E_0}}\right)$ is the geometric pair for $\gr(A)$.

Let $E_1 \subset E$ be those points with first coordinate nonzero. If $a_0 \neq 0$, then $b_0 \neq 0$ and so there is no loss in letting $a_0=b_0=1$. Hence, $e_i(p,q)=0$ implies $a_i=b_i$ and so we define $\sigma{\mid_{E_1}} = \id_{E_1}$. 

We define the scheme $E=E_0 \cup E_1 \subset \mathbb{P}(V^*)$ where $E_0$ corresponds to the point scheme of $\gr(A)$ and $E_1$ corresponds to the diagonal of $\van(R)$, that is, $p \in E_1$ if $(p,p) \in \van(R)$. Define the automorphism $\sigma$ where $\sigma{\mid_{E_0}}$ is the automorphism corresponding to $\gr(A)$ and $\sigma{\mid_{E_1}}=\id_{E_1}$. Thus, $E_1$ and $E_2$ are $\sigma$-invariant and $\van(R)$ is the graph of $E$. It is left to check that \textbf{G2} holds.

Let $F$ be $R$ reduced to $E_0$. If $f \in F$, then $f(p,p)=0$ for all $p \in E_1$, so $f$ corresponds to a relation in commutative affine space. Define $C=\K[x_1,\hdots,x_n]/(F)$. If $g \in V \tensor_k V$ such that $g(p,\sigma(p))=0$ for all $p\in E_1$, then $\hat{g}:=g{\mid_{E_1}} \in \van(F)$. By the Nullstellensatz, $\hat{g}^n \in F$ for some $n$. On the other hand, if $p \in E_0$, then either $g$ is a commutation relation or else $\gr(g) \in S$. Thus, $g$ is quadratic and so $n=1$. Hence, $g \in R$.\end{proof}

If $H$ is geometric, then $H_1 \approx \B_1$ and so there is a surjection $\mu:H \rightarrow \B$. When $H$ is noetherian, $I=\ker\mu$ is finitely generated by homogeneous elements and so there is hope of pulling information about $H$ back from $\B$. Let $M$ be a finite-dimensional simple module of $H$. Then $M$ is either $I$-torsion or it is $I$-torsionfree. Those of the first type may be regarded as modules over $H/I \iso \B$. Those of the second type are not as tractable, though results from \cite{atv2} give us a complete picture in the case that $H$ is regular of dimension 3. Our goal is to generalize the following example to homogenizations of 2-dimensional essentially regular algebras that are not PI.

\begin{ex}If $H=\hsjp$, then $E_1 = \{(1:a:\pm i)\}$. The finite-dimensional simple modules of $\sjp$ are exactly of the form $\sjp/((x_1-a)\sjp + (x_2 \pm i)\sjp)$. They are all non-isomorphic.\end{ex}

The following conjecture is closely related to these results. 

\begin{conj}[Walton]Let $S$ be a PBW deformation of a Sklyanin algebra that is not PI. Then all finite-dimensional simple modules of $S$ are 1-dimensional.\end{conj}

While we cannot answer this conjecture in its entirety, we make progress towards the affirmative by showing that the modules which are torsion over the canonical map $H(S) \rightarrow \B$ are 1-dimensional.

Let $A$ be essentially regular and $H=H(A)$ geometric with geometric pair $(E,\sigma)$. If $f$ is a defining relation of $H$ and $p=(p_0: p_1: \cdots : p_n) \in E_1$, then $f(p,\sigma(p))=0$ implies $\sigma(p)=p$. Thus, we write $f(p,p)=0$ or, more simply, $f(p)=0$. This is equivalent to defining the module,
	\[M_p = H/((x_0-1)H + (x_1-p_1)H + \cdots + (x_n-p_n)H ).\]
Since $H$ acts on $M_p$ via scalars, then $M$ is 1-dimensional. Since $x_0-1 \in \Ann(M_p)$, then $M_p$ may be identified with the $A$-module $A/((x_1-p_1)A + \cdots + (x_n-p_n)A )$. By an abuse notation, we also call this $A$-module $M_p$. Conversely, if $M=\{v\}$ is a 1-dimensional (simple) $A$-module, then $A$ acts on $M$ via scalars, say $x_i.v = c_iv$, $c_i \in \K$, $i=1,\hdots,n$. By setting $x_0.v=v$, $M$ becomes an $H$-module. This action must satisfy the defining relations of $H$ and so setting $p_i=c_i$ gives $f(p)=0$. We have now proved the following.

\begin{prop}\label{dim1}Let $A$ be essentially regular. The $A$-module $M$ is $1$-dimensional if and only if $M \iso M_p$ for some $p \in E_1$. Moreover, if $M_p$ and $M_q$ are 1-dimensional simple modules of $A$, then $M_p\iso M_q$ if and only if $p=q$.\end{prop}

For $A$ essentially regular, we believe that certain conditions will imply that these are all of the finite-dimensional simple modules. In the following, we will show that this is the case when $A$ is $2$-dimensional essentially regular and not PI.

\begin{lem}\label{simps}If $A$ is essentially regular and $M$ is a finite-dimensional simple module of $A$, then $\Ann(M) \neq 0$.\end{lem}

\begin{proof}Let $H=H(A)$. Write $M=M_A$ (resp. $M=M_H$) when $M$ is regarded as an $A$-module (resp. $H$-module). If $N_H \subset M_H$ as an $H$-module, then $N_A \subset M_A$, so $N_H=0$ or $N_H=M_H$. Thus, $M_H$ is a simple module and, moreover, $\dim_A(M_A)=\dim_H(M_H)$. By \cite{walton}, Lemma 3.1, if $M_H$ is a finite-dimensional simple module and $P$ is the largest graded ideal contained in $\Ann(M_H)$, then $\gk(H/P)$ is 0 or 1. If $P=0$, then $\gk(H/P)=\gk(H)>1$ when $H$ is regular of dimension greater than 1. Hence, if $\dim(M_H)>1$, then $\Ann(M_H)\neq 0$. Since $x_0.m=m$, then $x_0-1 \in \Ann(M_H)$, but $x_0-1$ is not a homogeneous element so $x_0-1 \notin P$. Let $r \in P \subset \Ann(M_H)$ with $r\neq 0$. If $r \in \K[x_0]$ with $r \neq x_0-1$, then $1 \in P$ so $\Ann(M_H)=H$. Hence, $r \notin \K[x_0]$ and so $r \not\equiv 0 \mod (x_0-1)$. Thus, $\Ann(M_A)\neq 0$.\end{proof}

The following result is well-known.

\begin{prop}Suppose $A=\jp$ or $A=\qp$ with $q \in \K^\times$ a nonroot of unity. Then every finite-dimensional simple module of $A$ is 1-dimensional.\end{prop}

\begin{proof}(Sketch) Let $M$ be such a module and let $xy$ and $y$ be the standard generators of $A$. By Lemma 5.9, $\Ann(M)$ is a nonzero prime ideal. In the case $A=\qp$, one checks that $C=\{x^iy^j \mid i,j \in \mathbb{N}\}$ is an Ore set and $AC\inv$ is a simple ring, so that every prime ideals contains either $x$ or $y$. Now $A/xA \iso \K[y]$ and $A/yA \iso \K[x]$ and the result follows. For $A=\jp$, one repeats with $C=\{y^i \mid i \in \mathbb{N}\}$ so that every nonzero prime ideal contains $y$.\end{proof}

The following result applies to homogenizations of 2-dimensional essentially regular algebras. However, in light of Proposition \ref{geo}, it seems reasonable that it may apply to certain higher dimensional algebras as well.

\begin{thm}\label{fdim}Let $A$ be an essentially regular algebra of dimension two that is not PI. If $M$ is a finite-dimensional simple $A$-module, then $M$ is 1-dimensional.\end{thm}

\begin{proof}Let $g \in H=H(A)$ be the canonical element such that $H/gH \iso \B = B(E,\mathcal{L},\sigma)$ and let $Q=\Ann M$. Because $|\sigma|=\infty$, the set of $g$-torsionfree simple modules of $H$ is empty (\cite{atv2}, Theorem 7.3). Hence, we may assume $M$ is $g$-torsion and therefore $M$ corresponds to a finite-dimensional simple module of $\B$.

Since $H$ is a homogenization, then $g=g_0g_1$ where $g_i \notin \K$ for $i=1,2$. It is clear that $x_0 \mid g$ so set $g_0=x_0$. Hence, $g_0 \in Q$ or $g_1 \in Q$ because $Q$ is prime. 

If $g_0$ and $g_1$ are irreducible, then the point scheme $E$ decomposes as $E=E_0 \cup E_1$. Thus, $M$ corresponds to a finite-dimensional simple module of $\B(E_0,\mathcal{L},\sigma{\mid_{E_0}})$ or $\B(E_1,\L,\sigma{\mid_{E_1}})$. In the first case, we have that $\B(E_0,\L,\sigma{\mid_{E_0}})$ is isomorphic to the twisted homogeneous coordinate ring of $\qp$ or $\jp$. Since $\sigma{\mid_{E_1}} = \id$, then $\B(E_1,\L,\sigma{\mid_{E_1}})$ is commutative. Hence, $H/Q$ is commutative and $Q$ contains $x_0-1$ so $M$ is a 1-dimensional simple module of $A$.

If $g$ divides into three linear factors $g_i$, $i=1,2,3$, then $\B/g_i\B$ is isomorphic to the twisted homogeneous coordinate ring of $\qp$ or $\jp$ for each $i$.\end{proof}

As a corollary, we recover a well-known result regarding the Weyl algebra.

\begin{cor}The first Weyl algebra $\fwa$ has no finite-dimensional simple modules.\end{cor}

\begin{proof}If $p \in E_1$, then $p=(1,a,b)$. The defining relation $f=x_1x_2-x_2x_1-1$ gives $f(p,p) = ab-ba-1 = 1 \neq 0$. Hence, the point scheme $E_1$ is empty.\end{proof}

More generally, suppose $A$ is a PBW deformation of a noetherian geometric algebra and $H=H(A)$. By Proposition \ref{geo}, $H$ is geometric. Let $(E,\sigma)$ be the geometric data associated to $H$ and let $I$ be the kernel of the canonical map $H \rightarrow \B(E,\L,\sigma)$. Let $E_1$ be the fixed points of $E$ and $E_0 = E\backslash E_1$. We say $F \subset E_0$ is \textit{reducible} if there exists disjoint and nonempty subschemes $F',F'' \subset F$ such that $F=F' \cup F''$ and $\sigma(F') \subset F'$, $\sigma(F'') \subset F''$. We say $F$ is \textit{reduced} if it is not reducible.

\begin{thm}\label{Itors}With the above notation. If $M$ is a finite-dimensional simple module of $H$ that is $I$-torsion, then $M$ is either a module over $\gr(A)$ or else $M$ is 1-dimensional.\end{thm}

\begin{proof}Let $M$ be an $I$-torsion simple module of $H$, so we may regard $M$ as a simple module of $\B$. Let $Q=\Ann(M)$ and so $Q \neq 0$ by Lemma \ref{simps}. If $P$ is the largest homogeneous prime ideal contained in $Q$, then $P$ corresponds to a reduced closed subscheme of $F \subset E$ and $\B/P \iso \B(F,\mathcal{O}_F(1),\sigma{\mid_F})$ (\cite{dmtwist}, Lemma 3.3). These subschemes are well-understood in this case, and so either $F$ corresponds to a subscheme in the twisted homogeneous coordinate ring associated to $\gr(A)$, or else $F$ is fixed pointwise by $\sigma{\mid_F}$, in which case $\B/P$ is commutative.\end{proof}

If $S$ is a deformed Sklyanin algebra that is not PI, then the only finite-dimensional simple module over $\gr(S)$ is the trivial one (\cite{walton}, Theorem 1.3). By \cite{bruyn}, there are exactly 8 fixed points in $E$. Hence, all $I$-torsion, finite-dimensional simple modules are 1-dimensional.

The algebra $U(\mathfrak{sl}_2(\K))$ is essentially regular of dimension three, is not PI, but does have finite-dimensional simple modules of every dimension $n$. There are other examples of essentially regular algebras exhibiting the same behavior (see \cite{redman}, \cite{benkroby}). This leads to the following conjecture.

\begin{conj}Let $A$ be essentially regular of dimension three that is not PI. Then either all finite-dimensional simple modules are 1-dimensional or else $A$ has finite dimensional simple modules of arbitrarily large dimension.\end{conj}

We end this section with a brief foray into the PI case. Suppose $A$ is prime PI and essentially regular. By Proposition \ref{hprops}, $H=H(A)$ is also prime PI. Moreover, if we let $Q_A$ be the quotient division ring of $A$ and $Q_H$ that of $H$, then
	\[ \pid(H)=\pid(Q_H) = \pid(Q_A(x_0)) = \pid(A[x_0]) = \pid(A).\]
If $A$ is 2-dimensional essentially regular and PI, then $A=\wa$ or $A=\qp$ with $q$ a primitive $n$th root of unity. In each case, $\pid(H(A))=n$. One can also show that $n=|\sigma|$ where $\sigma$ is the automorphism of the geometric pair corresponding to $H(A)$. The proof of Proposition \ref{fdim} implies that the $g$-torsion finite-dimensional simple modules of $A$ correspond to the finite-dimensional simple modules of $\qp$.

The $g$-torsionfree simple modules of $H$ are in 1-1 correspondence with those of $H[g\inv]$. Let $\Lambda_0$ be its degree 0 component. Since $H$ contains a central homogeneous element of degree 1, then $\pid(\Lambda_0) = \pid(H) = n$ (\cite{leb}, page 149). Thus, by \cite{walton}, Theorem 3.5, the $g$-torsionfree simple modules of $H$ all have dimension $n$. 

\section{A 5-dimensional family of regular algebras}\label{fivedim}

Suppose $A$ and $B$ are regular. In terms of generators and relations, the algebra $C=A \tensor B$ is easy to describe. Let $\{x_i\}$ be the generators for $A$ and $\{y_i\}$ those for $B$. Let $\{f_i\}$ be the relations for $A$ and $\{g_i\}$ those for $B$. Associate $x_i \in A$ with $x_i \tensor 1 \in A \tensor B$, and similarly for the $y_i$. Then $A \tensor B$ is the algebra on generators $\{x_i,y_i\}$ satisfying the relations $\{f_i,g_i\}$ along with the relations $x_iy_j-y_jx_i=0$ for all $i,j$.

A similar description holds when $A$ and $B$ are essentially regular. By comparing global dimension, one sees that $H(A \tensor B) \niso H(A) \tensor H(B)$. However, a related identity will be used to prove the following proposition.

\begin{prop}Let $A$ and $B$ be essentially regular algebras. Then $A \tensor B$ is essentially regular.\end{prop}

\begin{proof}We must show that $H(A \tensor B)$ is regular given that $H(A)$ and $H(B)$ are. Suppose $z_0$ is the homogenizing element in $H(A \tensor B)$ and $x_0$, $y_0$ those in $H(A)$ and $H(B)$, respectively. By Proposition \ref{grprop} and \cite{maowu}, Proposition 3.5, it suffices to prove the following:
\begin{align}\label{teniso}
    H(A \tensor B)/z_0H(A \tensor B) \iso H(A)/x_0H(A) \tensor H(B)/y_0H(B).
\end{align}
This is clear from the defining relations for the given algebras.
\end{proof}

\begin{cor}\label{tencor}Let $A$ and $B$ be $2$-dimensional essentially regular. Then $H(A \tensor B)$ is regular of dimension five.\end{cor}

Using the techniques developed above, we hope to understand the module structure of algebras of the form $H(A \tensor B)$.

\begin{ex}Let $A=B=\jp$ with generating sets $\{x_1,x_2\}$ and $\{y_1,y_2\}$, respectively. Let $\hat{x}_i=x_i \tensor 1$ and $\hat{y}_i=1 \tensor y_i$ for $i=1,2$. Then $C = A \tensor B$ is generated by $\{\hat{x}_1,\hat{x}_2,\hat{y}_1,\hat{y}_2\}$ and the defining relations are
\begin{align*}
    f &= \hat{x}_1\hat{x}_2-\hat{x}_2\hat{x}_1+\hat{x}_1^2, \\
    g &= \hat{y}_1\hat{y}_2-\hat{y}_2\hat{y}_1+\hat{y}_1^2, \\
    h_{ij} &= \hat{x}_i\hat{y}_j-\hat{y}_j\hat{x}_i \text{ for } i,j \in \{1,2\}.
\end{align*}
Let $E^A,E^B,E^C$ be the point schemes of $A,B$ and $C$, respectively. We claim that $E^C \iso E^A \cup E^B$. Let $p=(a_1:a_2:a_3:a_4) \in \mathbb{P}^3$. Then $p \in E^C$ if there exists $q=(b_1:b_2:b_3:b_4) \in \mathbb{P}^3$ such that $(p,q)$ is a zero for the above defining relations. The relation $f_1$ gives $\frac{a_1}{a_2}=\frac{b_1}{b_1+b_2}$ and $f_2$ gives $\frac{a_3}{a_4}=\frac{b_3}{b_3+b_4}$. Substituting into the the additional relations gives $a_3=a_4=0$ or else $a_1=a_2=0$. In the first case the points correspond to $E_A$ and otherwise to $E_B$.\end{ex}

The following proposition generalizes the above example.

\begin{prop}Suppose $A$ and $B$ are regular and $C=A \tensor B$ is not commutative. Then $E^C = E^A \cup E^B$.\end{prop}

\begin{proof}Let $\{x_1,\hdots,x_n,y_1,\hdots,y_m\}$ be the generators of $C$ subject to relations $\{f_i$, $g_i$, $h_{ij}\}$ such that the subalgebra generated by the $x_i$ (resp. $y_i$) subject to the relations $f_i$ (resp. $g_i$) is isomorphic to $A$ (resp. $B$). Identify $A$ and $B$ with their respective images in $C$. Let $h_{ij}=x_iy_j-y_jx_i$ for $1 \leq i \leq n,1 \leq j \leq m$. Let $E$ be the point scheme of $C$ and $\sigma$ the corresponding automorphism.  Choose $p = (a_1:\cdots:a_n:b_1:\cdots:b_m) \in E_A \times E_B$ and $q=\sigma(p)=(c_1:\cdots:c_n:d_1:\cdots:d_m)$. We claim either $a_i=0$ for all $i$ or else $b_i=0$ for all $i$.

Let $\sigma_A=\sigma{\mid_A}$ and $\sigma_B=\sigma{\mid_B}$. Suppose there exists $l,k$ such that $a_l\neq 0$ and $b_k \neq 0$. There is no loss in letting $a_l=1$. Hence, $0=h_{lj}(p,q)=d_j-b_jc_l$. If $c_l=0$, then $d_j=0$ for all $j$. Hence, $\sigma_B (b_1:\hdots:b_m)=0$, a contradiction, so $c_l\neq 0$. Then $b_j=d_j$ for all $j$. Thus, either $a_i=0$ for all $i$ or else $\sigma_B$ is constant, so $B$ is commutative. An identical argument shows that either $b_i=0$ for all $i$ or else $\sigma_A$ is constant, so $A$ is commutative. If $A$ and $B$ are commutative, then so is $C$.\end{proof}

If $A$ and $B$ are essentially regular, then the point scheme of $H(A \tensor B)$ has two components, $E_0$ and $E_1$. The component $E_0$ corresponds to that of $H(A \tensor B)/z_0H(A \tensor B)$ (see (\ref{teniso})). An argument similar to that from the previous proposition shows that $E_1$ corresponds to $E_1^A \cup E_1^B$. Consequently, if $M$ is a 1-dimensional simple module of $A \tensor B$, then $M$ is isomorphic to a 1-dimensional simple module of $A$ or $B$.

\section*{Acknowledgements}

The author would like to thank his advisor, Allen Bell, for his assistance throughout this project. Additionally, the author thanks to Dan Rogalski for helpful conversations at the MSRI Summer Graduate Workshop on Noncommuatative Algebraic Geometry and correspondence afterwards.

\bibliography{biblio}{}
\bibliographystyle{plain}

\end{document}